\documentclass[10pt]{book}
\usepackage[sectionbib]{natbib}
\usepackage{array,epsfig,fancyheadings,rotating}
\usepackage{color,soul}

\usepackage{amsmath}
\usepackage{amssymb}
\usepackage{amsfonts}
\usepackage{multirow}
\usepackage{amsthm}

\newtheorem{corollary}{Corollary}
\newtheorem{proposition}{Proposition}
\theoremstyle{definition}
\newtheorem{definition}{Definition}

\newtheorem{remark}{Remark}
\newcommand{\field}[1]{\mathbb{#1}}

\newcommand{\Z}{\field{Z}}

\newcommand{\design}{{\mathcal D}}
\newcommand{\fraction}{{\mathcal F}}

\DeclareMathOperator*{\xgcd}{gcd}
\DeclareMathOperator*{\xlcm}{lcm}

\begin{document}

\def\thefigure{\arabic{figure}}
\def\thetable{\arabic{table}}

\fontsize{10.95}{14pt plus.8pt minus .6pt}\selectfont

\setcounter{chapter}{1}
\setcounter{equation}{0} 


\renewcommand{\baselinestretch}{1.2}

\markright{ \hbox{\footnotesize\rm 
}\hfill\\[-13pt]
\hbox{\footnotesize\rm
}\hfill }

\markboth{\hfill{\footnotesize\rm Roberto Fontana, Fabio Rapallo and Maria Piera Rogantin} \hfill}
{\hfill {\footnotesize\rm Aberrations in multilevel designs} \hfill}

\renewcommand{\thefootnote}{}
$\ $\par


\fontsize{10.95}{14pt plus.8pt minus .6pt}\selectfont
\vspace{0.8pc}
\centerline{\large\bf ABERRATION IN QUALITATIVE MULTILEVEL DESIGNS}
\vspace{.4cm}
\centerline{Roberto Fontana$^{1}$, Fabio Rapallo$^{2}$ and Maria Piera Rogantin$^{3}$}
\vspace{.4cm}
\centerline{\it $^{1}$Politecnico di Torino, Italy, \ $^{2}$Universit\`a del Piemonte Orientale, Italy, }
\centerline{\it and $^{3}$Universit\`a di Genova, Italy}
\vspace{.55cm}
\fontsize{9}{11.5pt plus.8pt minus .6pt}\selectfont


\begin{quotation}
\noindent {\it Abstract:} Generalized Word Length Pattern (GWLP) is an important and widely-used tool for comparing fractional factorial designs. We consider qualitative factors, and we code their levels using the roots of the unity. We write the GWLP of a fraction ${\mathcal F}$ using the polynomial indicator function, whose coefficients encode many properties of the fraction. We show that the coefficient of a simple or interaction term can be written using the counts of its levels. This apparently simple remark leads to major consequence, including a convolution formula for the counts. We also show that the \emph{mean aberration} of a term over the permutation of its levels provides a connection with the variance of the level counts. Moreover, using \emph{mean aberrations} for symmetric $s^m$ designs with $s$ prime, we derive a new formula for computing the GWLP of $\fraction$. It is computationally easy, does not use complex numbers and also provides a clear way to interpret the GWLP. As case studies, we consider non-isomorphic orthogonal arrays that have the same GWLP. The different distributions of the \emph{mean aberrations} suggest that they could be used as a further tool to discriminate between fractions.\par

\vspace{9pt}
\noindent {\it Key words and phrases:} Algebraic Statistics, complex coding, counts of factor levels, fractional factorial designs, generalized word-length pattern, indicator function
\par
\end{quotation}\par

\section{Introduction}
In design of experiments, Generalized Word-Length Pattern (GWLP)  is an important tool for comparing fractional factorial designs. For a regular fraction ${\mathcal F}$ of a full-factorial design ${\mathcal D}$ with $m$ factors, the  Word Length Pattern (WLP) of ${\mathcal F}$ has been introduced by \cite{suen|chen|wu:97} as the sequence $A({\mathcal F})=(A_1({\mathcal F)}, A_2({\mathcal F}), \ldots, A_m({\mathcal F})$), where $A_j$ is the number of defining words with length $j$. Such a measure of the degree of aliasing can be easily interpreted in the regular case. The GWLP has been generalized for non-regular asymmetrical designs by \cite{xu|wu:01}, but it has a less evident meaning than in the regular case.

The aberration and the GWLP through the polynomial indicator function of the fraction $\fraction$ have been introduced in \cite{li|lin|ye:03} and \cite{cheng|ye:04} for two- and three-level cases respectively. In those papers the aberration of a simple or interaction term is defined as the square of the module of the corresponding coefficient of the indicator function, and the $j$-th element $A_j({\mathcal F})$ of the GWLP is the sum of the aberrations of the terms of order $j$, $j=1,\dots.m$.

As demonstrated in \cite{xu|wu:01}, the GWLP does not depend on the choice of a particular orthonormal basis of the functions defined over ${\mathcal D}$, while the aberration does. \cite{pistone|rogantin:08} use the complex coding of the factor levels to express the basis of the functions, and in particular of the indicator function. With this coding the coefficients of the indicator function are related in a simple manner to many interesting properties of the fraction and allows us to define aberration and GWLP in a clear way. The complex coding is particularly useful in the case of qualitative factors, as assumed in this work.
To simplify the computation, avoiding the use of complex numbers, \cite{fontana|pistone:13} represent the coefficients using the counts of the levels appearing in each simple or interaction term. As general references for GWLP and its properties, the reader can refer to \cite{mukerjee|wu:06} and \cite{chen|cheng:12}.

The practical use of the GWLP to discriminate among different designs is well known. Given two designs ${\mathcal F}_1$ and ${\mathcal F}_2$, the Generalized Minimum Aberration (GMA) criterion consists in the sequential minimization of the GWLP. ${\mathcal F}_1$ is better than ${\mathcal F}_2$ if there exists $j$ such that $A_1({\mathcal F}_1) = A_1({\mathcal F}_2), \ldots , A_j({\mathcal F}_1) = A_j({\mathcal F}_2)$ and $A_{j+1}({\mathcal F}_1) < A_{j+1}({\mathcal
F}_2)$. Despite the fact that the GMA criterion is widely applied, the statistical meaning of the elements of the GWLP is somewhat unclear. In the original work of \cite{xu|wu:01}, the GWLP of a symmetrical design is written as the MacWilliams transform of the distance distribution. This result has been generalized in \cite{qin|ai:07} to the case of multilevel designs. Under a different point of view, \cite{gromping|xu:14} writes the first non-zero element of the
GWLP as the sum of the $R^2$ coefficients of suitably defined linear models.  The connection between GWLP and Discrete Discrepancy has been investigated in \cite{qin|fang:04}.

In this work we use the expression of the GWLP via the aberrations of the interaction terms of a given order. In turn, the aberrations are computed using only the levels counts of the corresponding terms. We fully exploit such new expressions in two directions. First,  we establish a convolution formula for the counts of the terms, in symmetrical $s^m$ designs with $s$ a prime number. Second, we introduce the {\em mean aberration} of a simple or interaction term, over the permutations of its levels. The mean aberration has a very simple expression, and it is easy to compute and to explain. Indeed, we show that the mean aberration is proportional to the variance of the level counts.

Moreover, we prove that for symmetrical $s^m$ designs, $s$ prime, the $j$-th element $A_j({\mathcal F})$ of GWLP is the sum of the mean aberrations of the terms of order $j$, $j=1,\dots,m$ and therefore the mean aberrations produce an alternative decomposition of the GWLP. In our knowledge, the proposed formula is the simplest over all alternative expressions in literature. Nevertheless, in general, this property does not hold.

The paper is organized as follows. In Section \ref{sec:alg} a short review of the algebraic theory of factorial designs is given. In Section \ref{sec:conv} the convolution formula that expresses the relationships among the level counts is obtained, while in Section \ref{sec:GWLP} the mean aberration is defined, and its properties and its connection with the GWLP are studied. Section \ref{sec:choice} is devoted to the comparison of fractions that have the same GWLP but different distributions of the mean aberrations. Finally in Section \ref{sec:future} we briefly describe some directions for future work.

\section{Algebraic characterization of fractional designs} \label{sec:alg}

In this section, for ease in reference, we present some relevant results of the algebraic theory of fractional designs. The interested reader can find further information, including the proofs of the propositions, in \cite{fontana|pistone|rogantin:00} and \cite{pistone|rogantin:08}.

Let us consider an experiment which includes $m$ factors.

Let us code the $s_j$ levels of the $j$-th factor  by  the $s_j$-th roots of the unity  $\omega_k^{(s_j)}=\exp\left(\sqrt{-1}\:  \frac {2\pi}{s_j} \ k\right)$, $k=0,\ldots,s_j-1, \ j=1,\ldots,m$. We denote such a factor by $\Omega_{s_j}$, $\Omega_{s_j}=\left\{\omega_0,\ldots,\omega_{{s_j}-1}\right\}$.

As $\alpha=\beta \mod s$ implies $\omega_k^\alpha = \omega_k^\beta$, it is useful to introduce the residue class ring $\mathbb Z_s$ and the notation $[k]_s$ for the residue of $k \mod s$. For integer $\alpha$, we obtain $(\omega_k)^\alpha = \omega_{[\alpha k]_s}$. We also have $\omega_h\omega_k=\omega_{[h+k]_s}$. We drop the sub-$s$ notation when there is no ambiguity.

We denote by   $\design$  the full factorial design with complex coding:
\begin{equation*}
\design = \design_1 \times \cdots \design_j \cdots \times \design_m \ \textrm{ with } \design_j=\Omega_{s_j} \ .
\end{equation*}
the cardinality of the full factorial design is $\# \design=\prod_{j=1}^m s_j$.

We denote by $L$ the exponent set of the complex coded design $\{0,\ldots,s_j-1\}$, $j=1,\dots,m$:
\begin{equation*}
L  = \mathbb Z_{s_1} \times \cdots  \times \mathbb Z_{s_m} \ .
\end{equation*}
Notice that $L$ is both  the  exponent set of the complex coded design and the integer coded design. The elements of $L$ are denoted by $\alpha$, $\beta, \ldots$:
\begin{equation*}
L = \left\{ \alpha= (\alpha_1,\ldots,\alpha_m) : \alpha_j = 0,\ldots,s_j-1, j=1,\ldots,m\right\} \ ;
\end{equation*}
$[\alpha-\beta]$ is the $m$-tuple $\left(\left[\alpha_1-\beta_1 \right]_{s_1}, \ldots, \left[\alpha_j-\beta_j \right]_{s_j}, \ldots, \left[\alpha_m - \beta_m\right]_{s_m} \right)$. The computation of the $j$-th element is in the ring $\mathbb Z_{s_j}$.

In order to use polynomials to represent all the functions defined over $\design$, including counting functions, we define
\begin{itemize}
\item $X_j$, the $j$-th component function, which maps a point $\zeta=(\zeta_1,\ldots,\zeta_m)$ of $\design$ to its $j$-th component,
\begin{equation*}
X_j \colon \design \ni (\zeta_1,\ldots,\zeta_m)\ \longmapsto \ \zeta_j \in \design_j \ .
\end{equation*}
The function $X_j$ is a \emph{simple term} or, by abuse of terminology, a \emph{factor}.
\item $X^\alpha=X_1^{\alpha_1} \cdot \ldots \cdot X_m^{\alpha_m}$, $\alpha \in L =  \Z_{s_1} \times \cdots \times  \Z_{s_m}$ i.e., the monomial function
\begin{equation*}
X^\alpha : \design \ni (\zeta_1,\ldots,\zeta_m)\ \mapsto \ \zeta_1^{\alpha_1}\cdot \ldots \cdot \zeta_m^{\alpha_m} \ .
\end{equation*}
The function $X^\alpha$ is an \emph{interaction term}.
\end{itemize}

The following proposition provides the level set of a term $X^\alpha$, for any choice $s_1, \ldots, s_m$ of the levels of the factors.

\begin{proposition} \label{pr:P-1-2}
On the full factorial design $\design$, the simple term $X_j^{\alpha_j}$ takes values in $\Omega_{t_j}$, where $t_j=s_j/\xgcd(\alpha_j,s_j)$. The interaction term $X^\alpha$ takes values in $\Omega_{t_\alpha}$ where ${t_\alpha}=\xlcm(t_1,\ldots,t_m)$ and $t_i$ ($i=1,\ldots,m$) are determined as before.
\end{proposition}
\begin{proof}
It follows from the properties of the arithmetic in ${\mathbb Z}_s$.
\end{proof}

We observe that for a symmetric $s^m$ design with $s$ prime number, $t=s$ for all $\alpha \in L$.

The set of monomials $\{X^\alpha: \alpha \in L\}$ is an orthonormal basis of all the complex functions defined over $\design$.

\begin{definition}
A fraction $\fraction$ is a multiset $(\fraction_*,R)$ whose underlying set of elements $\fraction_*$ is contained in $\design$ and $R$ is the multiplicity function $R: \fraction_* \rightarrow \mathbb N$ that for each element in $\fraction_*$ gives the number of times it belongs to the multiset $\fraction$. We call $R$  \emph{counting function}.
\end{definition}
The underlying set of elements $\fraction_*$, referred to as the support of $\fraction_*$, is the subset of $\design$ that contains all the elements of $\design$ that appear in $\fraction$ at least once. We denote the number of elements of a fraction $\fraction$ by $n$, with $n = \sum_{\zeta \in \fraction_*} R(\zeta)$.

We use the basis $\{X^\alpha: \alpha \in L\}$ to represent the counting function of a fraction.

\begin{definition} \label{de:indicator}
The \emph{counting function} $R$ of a fraction $\fraction$ is a complex polynomial defined over $\design$ so that for each $\zeta \in \design$, $R(\zeta)$ equals the number of appearances of $\zeta$ in the fraction. A $0-1$ valued counting function is called \emph{indicator function} of a single replicate fraction $\fraction$.
We denote by $c_\alpha$ the coefficients of the representation of $R$  on $\design$ using the monomial basis
$\{X^\alpha, \ \alpha \in L\}$:
\begin{equation*}
R(\zeta) = \sum_{\alpha \in L} c_\alpha X^\alpha(\zeta), \;\zeta\in\design, \;  c_\alpha \in \mathbb C \ .
\end{equation*}
\end{definition}

\begin{proposition} \label{pr:ceqn}
Let $\fraction$ be a fraction with  counting function $R$.
\begin{enumerate}
\item
The coefficients $c_\alpha$ of  $R$  are given by:
\begin{equation} \label{calpha}
c_\alpha = \frac 1 {\#\design} \sum_{\zeta \in \fraction} X^{[-\alpha]}(\zeta) =
 \frac{1}{\#\design} \sum_{h=0}^{t_{\alpha}-1} n_{\alpha,t_{\alpha}-h} \ \omega_{h}
\end{equation}
where $t_{\alpha}$ is determined by $X^{\alpha}$ according to Proposition \ref{pr:P-1-2} and $n_{\alpha,h}$ is the number of the occurrences of  $\omega_h$ in $\{X^{\alpha}(\zeta):\zeta \in \fraction\}$.

In particular $c_0= {n}/{\#\design}$.
\item
If $\fraction$ is a single replicate fraction, the coefficients $c_\alpha$ satisfy the following relationships:
\begin{equation}\label{conv-c}
c_\alpha= \sum_{\beta\in L}  c_{\beta}\ c_{[\alpha-\beta]}
\end{equation}
for each $\alpha \in L$.
\end{enumerate}
\end{proposition}
\begin{proof}
\begin{enumerate}
\item
We have
\begin{equation*}
c_\alpha =\frac 1 {\#\design} \sum_{\zeta \in \fraction} X^{[-\alpha]}(\zeta) =
\frac{1}{\#\design} \sum_{h=0}^{t_{\alpha}-1} n_{[-\alpha],h} \ \omega_{h}=\frac{1}{\#\design} \sum_{h=0}^{t_{\alpha}-1} n_{\alpha,t_\alpha-h} \ \omega_{h}
\end{equation*}
where the first equality is proved in \cite{pistone|rogantin:08}, the second equality is proved in \cite{fontana|pistone:13}, and the last one derives from properties of the roots of the unity.
\item See \cite{pistone|rogantin:08}.
\end{enumerate}
\end{proof}

The coefficients $c_\alpha$ encode many interesting properties of the fraction $\fraction$ as orthogonality among factors and interactions, regularity, and aberration, see \cite{pistone|rogantin:08}.

\section{Convolution formula with counts}\label{sec:conv}

It is well known that the simple and interaction terms are dependent each other. The proposition below gives such relationships in term of the level counts of all terms on a fraction $\fraction$.

\begin{proposition} \label{pr:convol}
Let $\fraction$ be a single replicate fraction of a symmetric $s^m$ designs $\design$, with $s$  prime number. The counts $n_\alpha$, $n_\alpha=(n_{\alpha,0},\dots,n_{\alpha,s-1})$, are related according to
\begin{equation}\label{eq:conv}
\sum_{k=0}^{s-1} \left( \#\design \ n_{\alpha,k} - \sum_{\beta\in L} \sum_{i=0}^{s-1} n_{\beta,i}\ n_{[\alpha-\beta],[k-i]}\right)\omega_k=0
\end{equation}
for each $\alpha \in L$. For convenience we define $n_{0,k}=0$ for $k=1,\ldots,s-1$
\end{proposition}
\begin{proof}
From Equation \eqref{conv-c} we have
\begin{equation*}
c_{[s-\alpha]} = \sum_{\beta \in L} c_{[s-\beta]} \ c_{[\beta-\alpha]} \ .
\end{equation*}
Substituting the coefficients $c_{[s-\alpha]}$ expressed in terms of the level counts $(n_{\alpha,0},\dots,n_{\alpha,s-1})$, as in Equation \eqref{calpha}, we have:
\begin{eqnarray*}
\sum_{k=0}^{s-1} n_{\alpha,k}\ \omega_k=\frac 1 {\# \design} \sum_{\beta \in L} \sum_{i=0}^{s-1} n_{\beta,i}\ \omega_i \sum_{j=0}^{s-1} n_{[\alpha-\beta],j} \ \omega_j \ .
\end{eqnarray*}
Taking $k=i+j$, the thesis follows.
\end{proof}

\begin{corollary}\label{corr:convol}
We denote by $r_{\alpha,k}$ the coefficients of $\omega_k$ in Equation \eqref{eq:conv}:
\begin{equation*}
r_{\alpha,k}= \#\design \ n_{\alpha,k} - \sum_{\beta\in L} \sum_{i=0}^{s-1} n_{\beta,i}\ n_{[\alpha-\beta],[k-i]} \, .
\end{equation*}
If $s$ is prime, for each $\alpha$, Proposition \ref{pr:convol}  gives $s-1$ relationships among counts:
\begin{equation*}
 r_{\alpha,0}=r_{\alpha,1}=\dots=r_{\alpha,s-1} \ .
 \end{equation*}
\end{corollary}
\begin{proof}
If $s$ is prime, a polynomial $\sum_{k=0}^{s-1} r_k \omega_k$ is zero if and only if $r_0=r_1=\dots=r_{s-1}$.
\end{proof}

We observe that previous relationships are not independent both because of fraction properties and complex numbers properties; for instance:
\begin{description}
\item [C1 ] $\sum_{k=0}^{t-1} n_{\alpha,k}=n$, for each $\alpha$;
\item [C2 ] $ n_{(0,\dots, 0),0}= n$ and $ n_{(0,\dots,0),k}= 0$, for  $k=1,\dots,t-1$;
\item [C3 ] $ n_{(a_1,a_2,\dots,a_m),k}=n_{([s_1-a_1],[s_2-a_2],\dots,[s_m-a_m]),[s_t-k]}$ for $(a_1,a_2,\dots,a_m) \in \{0,\dots,s_1-1\}\times\{0,\dots,s_2-1\}\times\dots\times\{0,\dots,s_m-1\}$ and $k \in   \{0,\dots,t-1\}$.
\item [C4 ] possible relationships coming from desired properties of the fraction (for instance, the fact that a term is centered and/or some terms are  mutually orthogonal, see Proposition 3 of \cite{pistone|rogantin:08}).
\end{description}

In this a way the number of variables involved in the problem is reduced and  the actual computation of all admissible patterns of a fraction can be simplified, as shown in the following examples.

We consider orthogonal arrays of strength $2$, i.e., with balanced simple and second order interactions.
\begin{itemize}
\item Fractions of a $3^3$ design with $9$ runs.

The fractions have $9$ runs.  For $k=0,1,2$, and $a,b,c=1,2$, we have:
\begin{itemize}
\item[$\diamondsuit$] $\sum_{k=0}^2 n_{(a,b,c),k}=9$  ,
\item[$\diamondsuit$] $ n_{(0,0,0),0}= 9$ \ , \quad  $n_{(0,0,0),1}=n_{(0,0,0),2}= 0 $ ,
\item[$\diamondsuit$] $ n_{(2,[3-a],[3-b]),[3-k]}= n_{(1,a,b),k}$
\item[$\diamondsuit$] $ n_{(a,0,0),k}=n_{(0,a,0),k}= n_{(0,0,a),k}=3$  \ , $n_{(a,b,0),k}= n_{(a,0,b),k}=n_{(0,a,b),k}=3$
\end{itemize}

The relationships of Corollary \ref{corr:convol}, computed with CoCoA (see \cite{CoCoA-5}), are:
\begin{align*}
&n_{(1,1,1),0}^2 - n_{(1,1,1),1}^2 - 2n_{(1,1,1),0}n_{(1,1,1),2} +
        2n_{(1,1,1),1}n_{(1,1,1),2} - 9n_{(1,1,1),0} + 9n_{(1,1,1),1},\\
&-2n_{(1,1,1),0}n_{(1,1,1),1} + n_{(1,1,1),1}^2 + 2n_{(1,1,1),0}
        n_{(1,1,1),2} - n_{(1,1,1),2}^2 - 9n_{(1,1,1),1} + 9n_{(1,1,1),2}
\end{align*}

In order to easily handle these polynomials we compute the Gr\"obner basis of corresponding ideal:
\begin{equation*}
\begin{array}{ll}
n_{(1,2,2),0} + n_{(1,2,2),1} + n_{(1,2,2),2} - 9, & \quad
n_{(1,2,1),0} + n_{(1,2,1),1} + n_{(1,2,1),2} - 9, \\
n_{(1,1,2),0} + n_{(1,1,2),1} + n_{(1,1,2),2} - 9, & \quad
n_{(1,1,1),0} + n_{(1,1,1),1} + n_{(1,1,1),2} - 9,\\
n_{(1,1,1),1}^2 - n_{(1,1,1),2}^2 - 9n_{(1,1,1),1} + 9n_{(1,1,1),2}, &\quad n_{(1,1,1),2}^3 - 12n_{(1,1,1),2}^2 + 27n_{(1,1,1),2},\\
n_{(1,1,1),1}n_{(1,1,1),2} + 1/2n_{(1,1,1),2}^2 - 9/2n_{(1,1,1),2} \ . &
\end{array}
\end{equation*}

The admissible configurations for $n_{(1,1,1)}$ are:
\begin{equation*}
(3,3,3) \ , \quad (9,0,0) \ , \quad (0,9,0) \ , \quad (0,0,9) \ .
\end{equation*}
The counts for the other $\alpha$ follow from the constraints C1-C4 listed above.

\item Fractions of a $5^3$ design with $25$ runs.
In this case the computation is heavier than in the previous example. The constraints C1-C4 give 32 polynomials in 80 variables (the counts), one of which is shown below.

{\tiny{
\begin{equation*}
\begin{array}{l}
n_{(1,1,1),0}\ n_{(1,2,2),0} - n_{(1,1,1),4}\ n_{(1,2,2),0} - n_{(1,1,1),0}\ n_{(1,2,2),1} + n_{(1,1,1),1}\ n_{(1,2,2),1} - n_{(1,1,1),1}\ n_{(1,2,2),2}
\\+ n_{(1,1,1),2}\ n_{(1,2,2),2} - n_{(1,1,1),2}\ n_{(1,2,2),3} + n_{(1,1,1),3}\ n_{(1,2,2),3} - n_{(1,1,1),3}\ n_{(1,2,3),4} + n_{(1,1,1),4}\ n_{(1,2,2),4} \\+ n_{(1,1,2),0}\ n_{(1,2,3),0} - n_{(1,1,2),4}\ n_{(1,2,3),0} - n_{(1,1,2),0}\ n_{(1,2,3),1} + n_{(1,1,2),1}\ n_{(1,2,3),1} - n_{(1,1,2),1}\ n_{(1,2,3),2} \\+ n_{(1,1,2),2}\ n_{(1,2,3),2} - n_{(1,1,2),2}\ n_{(1,2,3),3} + n_{(1,1,2),3}\ n_{(1,2,3),3} - n_{(1,1,2),3}\ n_{(1,2,3),4} + n_{(1,1,2),4}\ n_{(1,2,3),4} \\+ n_{(1,1,3),0}\ n_{(1,2,4),0} - n_{(1,1,3),4}\ n_{(1,2,4),0} - n_{(1,1,3),0}\ n_{(1,2,4),1} + n_{(1,1,3),1}\ n_{(1,2,4),1} - n_{(1,1,3),1}\ n_{(1,2,4),2} \\+ n_{(1,1,3),2}\ n_{(1,2,4),2} - n_{(1,1,3),2}\ n_{(1,2,4),3} + n_{(1,1,3),3}\ n_{(1,2,4),3} - n_{(1,1,3),3}\ n_{(1,2,4),4} + n_{(1,1,3),4}\ n_{(1,2,4),4}
\\
+ n_{(1,1,4),0}\ n_{(1,3,1),0} - n_{(1,1,4),3}\ n_{(1,3,1),0} + n_{(1,1,4),1}\ n_{(1,3,1),1} - n_{(1,1,4),4}\ n_{(1,3,1),1} - n_{(1,1,4),0}\ n_{(1,3,1),2} \\+ n_{(1,1,4),2}\ n_{(1,3,1),2} - n_{(1,1,4),1}\ n_{(1,3,1),3} + n_{(1,1,4),3}\ n_{(1,3,1),3} - n_{(1,1,4),2}\ n_{(1,3,1),4} + n_{(1,1,4),4}\ n_{(1,3,1),4} \\+ n_{(1,2,1),0}\ n_{(1,3,2),0} - n_{(1,2,1),4}\ n_{(1,3,2),0} - n_{(1,2,1),0}\ n_{(1,3,2),1} + n_{(1,2,1),1}\ n_{(1,3,2),1} - n_{(1,2,1),1}\ n_{(1,3,2),2} \\+ n_{(1,2,1),2}\ n_{(1,3,2),2} - n_{(1,2,1),2}\ n_{(1,3,2),3} + n_{(1,2,1),3}\ n_{(1,3,2),3} - n_{(1,2,1),3}\ n_{(1,3,2),4} + n_{(1,2,1),4}\ n_{(1,3,2),4} \\+ n_{(1,1,1),0}\ n_{(1,3,3),0} - n_{(1,1,1),3}\ n_{(1,3,3),0} + n_{(1,2,2),0}\ n_{(1,3,3),0} - n_{(1,2,2),4}\ n_{(1,3,3),0} + n_{(1,1,1),1}\ n_{(1,3,3),1} \\- n_{(1,1,1),4}\ n_{(1,3,3),1} - n_{(1,2,2),0}\ n_{(1,3,3),1} + n_{(1,2,2),1}\ n_{(1,3,3),1} - n_{(1,1,1),0}\ n_{(1,3,3),2} + n_{(1,1,1),2}\ n_{(1,3,3),2}
\\
- n_{(1,2,2),1}\ n_{(1,3,3),2} + n_{(1,2,2),2}\ n_{(1,3,3),2} - n_{(1,1,1),1}\ n_{(1,3,3),3} + n_{(1,1,1),3}\ n_{(1,3,3),3} - n_{(1,2,2),2}\ n_{(1,3,3),3} \\+ n_{(1,2,2),3}\ n_{(1,3,3),3} - n_{(1,1,1),2}\ n_{(1,3,3),4} + n_{(1,1,1),4}\ n_{(1,3,3),4} - n_{(1,2,2),3}\ n_{(1,3,3),4} + n_{(1,2,2),4}\ n_{(1,3,3),4} \\+ n_{(1,1,2),0}\ n_{(1,3,4),0} - n_{(1,1,2),3}\ n_{(1,3,4),0} + n_{(1,2,3),0}\ n_{(1,3,4),0} - n_{(1,2,3),4}\ n_{(1,3,4),0} + n_{(1,1,2),1}\ n_{(1,3,4),1} \\- n_{(1,1,2),4}\ n_{(1,3,4),1} - n_{(1,2,3),0}\ n_{(1,3,4),1} + n_{(1,2,3),1}\ n_{(1,3,4),1} - n_{(1,1,2),0}\ n_{(1,3,4),2} + n_{(1,1,2),2}\ n_{(1,3,4),2} \\- n_{(1,2,3),1}\ n_{(1,3,4),2} + n_{(1,2,3),2}\ n_{(1,3,4),2} - n_{(1,1,2),1}\ n_{(1,3,4),3} + n_{(1,1,2),3}\ n_{(1,3,4),3} - n_{(1,2,3),2}\ n_{(1,3,4),3} \\+ n_{(1,2,3),3}\ n_{(1,3,4),3} - n_{(1,1,2),2}\ n_{(1,3,4),4} + n_{(1,1,2),4}\ n_{(1,3,4),4} - n_{(1,2,3),3}\ n_{(1,3,4),4} + n_{(1,2,3),4}\ n_{(1,3,4),4}
\\
+ n_{(1,1,3),0}\ n_{(1,4,1),0} - n_{(1,1,3),2}\ n_{(1,4,1),0} + n_{(1,2,4),0}\ n_{(1,4,1),0} - n_{(1,2,4),3}\ n_{(1,4,1),0} + n_{(1,1,3),1}\ n_{(1,4,1),1} \\- n_{(1,1,3),3}\ n_{(1,4,1),1} + n_{(1,2,4),1}\ n_{(1,4,1),1} - n_{(1,2,4),4}\ n_{(1,4,1),1} + n_{(1,1,3),2}\ n_{(1,4,1),2} - n_{(1,1,3),4}\ n_{(1,4,1),2} \\- n_{(1,2,4),0}\ n_{(1,4,1),2} + n_{(1,2,4),2}\ n_{(1,4,1),2} - n_{(1,1,3),0}\ n_{(1,4,1),3} + n_{(1,1,3),3}\ n_{(1,4,1),3} - n_{(1,2,4),1}\ n_{(1,4,1),3} \\+ n_{(1,2,4),3}\ n_{(1,4,1),3} - n_{(1,1,3),1}\ n_{(1,4,1),4} + n_{(1,1,3),4}\ n_{(1,4,1),4} - n_{(1,2,4),2}\ n_{(1,4,1),4} + n_{(1,2,4),4}\ n_{(1,4,1),4} \\+ n_{(1,1,4),0}\ n_{(1,4,2),0} - n_{(1,1,4),2}\ n_{(1,4,2),0} + n_{(1,3,1),0}\ n_{(1,4,2),0} - n_{(1,3,1),4}\ n_{(1,4,2),0} + n_{(1,1,4),1}\ n_{(1,4,2),1} \\- n_{(1,1,4),3}\ n_{(1,4,2),1} - n_{(1,3,1),0}\ n_{(1,4,2),1} + n_{(1,3,1),1}\ n_{(1,4,2),1} + n_{(1,1,4),2}\ n_{(1,4,2),2} - n_{(1,1,4),4}\ n_{(1,4,2),2}
\\- n_{(1,3,1),1}\ n_{(1,4,2),2} + n_{(1,3,1),2}\ n_{(1,4,2),2} - n_{(1,1,4),0}\ n_{(1,4,2),3} + n_{(1,1,4),3}\ n_{(1,4,2),3} - n_{(1,3,1),2}\ n_{(1,4,2),3} \\+ n_{(1,3,1),3}\ n_{(1,4,2),3} - n_{(1,1,4),1}\ n_{(1,4,2),4} + n_{(1,1,4),4}\ n_{(1,4,2),4} - n_{(1,3,1),3}\ n_{(1,4,2),4} + n_{(1,3,1),4}\ n_{(1,4,2),4} \\+ n_{(1,2,1),0}\ n_{(1,4,3),0} - n_{(1,2,1),3}\ n_{(1,4,3),0} + n_{(1,3,2),0}\ n_{(1,4,3),0} - n_{(1,3,2),4}\ n_{(1,4,3),0} + n_{(1,2,1),1}\ n_{(1,4,3),1} \\- n_{(1,2,1),4}\ n_{(1,4,3),1} - n_{(1,3,2),0}\ n_{(1,4,3),1} + n_{(1,3,2),1}\ n_{(1,4,3),1} - n_{(1,2,1),0}\ n_{(1,4,3),2} + n_{(1,2,1),2}\ n_{(1,4,3),2} \\- n_{(1,3,2),1}\ n_{(1,4,3),2} + n_{(1,3,2),2}\ n_{(1,4,3),2} - n_{(1,2,1),1}\ n_{(1,4,3),3} + n_{(1,2,1),3}\ n_{(1,4,3),3} - n_{(1,3,2),2}\ n_{(1,4,3),3} \\+ n_{(1,3,2),3}\ n_{(1,4,3),3} - n_{(1,2,1),2}\ n_{(1,4,3),4} + n_{(1,2,1),4}\ n_{(1,4,3),4} - n_{(1,3,2),3}\ n_{(1,4,3),4} + n_{(1,3,2),4}\ n_{(1,4,3),4}
\\
+ n_{(1,1,1),0}\ n_{(1,4,4),0} - n_{(1,1,1),2}\ n_{(1,4,4),0} + n_{(1,2,2),0}\ n_{(1,4,4),0} - n_{(1,2,2),3}\ n_{(1,4,4),0} + n_{(1,3,3),0}\ n_{(1,4,4),0} \\- n_{(1,3,3),4}\ n_{(1,4,4),0} + n_{(1,1,1),1}\ n_{(1,4,4),1} - n_{(1,1,1),3}\ n_{(1,4,4),1} + n_{(1,2,2),1}\ n_{(1,4,4),1} - n_{(1,2,2),4}\ n_{(1,4,4),1} \\- n_{(1,3,3),0}\ n_{(1,4,4),1} + n_{(1,3,3),1}\ n_{(1,4,4),1} + n_{(1,1,1),2}\ n_{(1,4,4),2} - n_{(1,1,1),4}\ n_{(1,4,4),2} - n_{(1,2,2),0}\ n_{(1,4,4),2} \\+ n_{(1,2,2),2}\ n_{(1,4,4),2} - n_{(1,3,3),1}\ n_{(1,4,4),2} + n_{(1,3,3),2}\ n_{(1,4,4),2} - n_{(1,1,1),0}\ n_{(1,4,4),3} + n_{(1,1,1),3}\ n_{(1,4,4),3} \\- n_{(1,2,2),1}\ n_{(1,4,4),3} + n_{(1,2,2),3}\ n_{(1,4,4),3} - n_{(1,3,3),2}\ n_{(1,4,4),3} + n_{(1,3,3),3}\ n_{(1,4,4),3} - n_{(1,1,1),1}\ n_{(1,4,4),4} \\+ n_{(1,1,1),4}\ n_{(1,4,4),4} - n_{(1,2,2),2}\ n_{(1,4,4),4} + n_{(1,2,2),4}\ n_{(1,4,4),4} - n_{(1,3,3),3}\ n_{(1,4,4),4} + n_{(1,3,3),4}\ n_{(1,4,4),4}
\end{array}
\end{equation*}}}

In this example it is unfeasible to explicitly compute all the admissible configurations, nevertheless one can check whether a given set of counts is admissible or not.
\end{itemize}

\section{Generalized Word Length Pattern}\label{sec:GWLP}

As mentioned in the Introduction, the $j$-th element of the GWLP of a fraction can be computed as the sum of the aberrations of the interaction terms $X^\alpha$ of order $j$. In this section, we make use of the formula in Equation \eqref{calpha} to express the aberrations as functions of the level counts, and we study the mean aberration of a term over the level permutations. The analysis of the mean aberration is valid for generic asymmetric multilevel designs, while its use for the computation of the GWLP is limited to symmetric $s^m$ designs, with $s$ prime number.

\begin{definition} \label{wlp1}
Given an interaction $X^\alpha$ defined on a fraction $\fraction$ of the full factorial design $\design$, its \emph{aberration}, or degree of aliasing, $a_\alpha$ is given by the real number
\begin{equation*}
a_\alpha=  \frac{ \|c_{\alpha}\|_2^2 }{c_{0}^2}
\end{equation*}
where $\| x \|_2^2$ is square of the norm of the complex number $x$.

The GWLP $A(\fraction)=(A_1(\fraction),\ldots,A_m(\fraction))$ of a fraction $\fraction$  is defined as
\begin{equation*}
A_j(\fraction)= \sum_{\|\alpha \|_0 =j}  a_{\alpha} \quad j=1,\ldots,m \ ,
\end{equation*}
where $\|\alpha \|_0$ is the number of non-null elements of $\alpha$, i.e., the order of interaction of $X^\alpha$.
\end{definition}

The following proposition allows us to compute the aberration  $a_\alpha$ without using complex computation.
\begin{proposition}
Let $X^\alpha$ be a simple or interaction term with values in $\Omega_{t}$. Its aberration $a_\alpha$ is
\begin{equation}\label{al}
a_\alpha=\frac{1}{n^2}\left( \sum_{k=0}^{t-1}  \cos\left(\frac{2\pi}{t} k \right) \sum_{i=0}^{t-1} n_{\alpha,i}
n_{\alpha,[i-k]}  \right) \, .
\end{equation}
\end{proposition}
\begin{proof}
This formula follows from Definition \ref{wlp1} using Equation \eqref{calpha}.
\end{proof}

We observe that, if $t$ is odd, we get
\begin{equation*}
a_\alpha=\frac{1}{n^2}\left( \sum_{i=0}^{t-1}  n_{\alpha,i}^2+ 2 \sum_{k=1}^{(t-1)/2} \cos\left(\frac{2\pi}{t} k
\right)  \sum_{i=0}^{t-1}  n_{\alpha,i} n_{\alpha,[i-k]} \right) \, .
\end{equation*}
Moreover, for $t=2,3,4$ we obtain
\begin{equation*}
n^2 a_\alpha=
\begin{cases}
 n_0^2+n_1^2-2n_0n_1 &\mbox{ if } t=2 \\
 n_{0}^2+n_{1}^2+n_{2}^2- n_{0}n_{2}- n_{1}n_{0}- n_{2}n_{1} &\mbox{ if } t=3  \\
 n_{0}^2+n_{1}^2+n_{2}^2+n_{3}^2-2n_{0}n_{2}-2n_{1}n_{3} &\mbox{ if } t=4
\end{cases}
\end{equation*}
where the suffix $\alpha$ is omitted to simplify the notation.

It is well known that if a term $X^\alpha$ is balanced on a fraction (i.e., $n_{\alpha,k}=n/t$ for all $k=1,\dots,t-1$), then $a_\alpha = 0$. Equation \eqref{al} gives an alternative proof of this fact. However, two disadvantages in the definition of $a_\alpha$ come out when facing with qualitative factors: (a) the aberration $a_\alpha$ is not independent on the permutation of the levels; (b) the fact that $a_\alpha=0$ does not guarantee that term $X^\alpha$, with $t$ levels, is balanced if $t$ is not prime.

To illustrate such problems, let us consider a term $X^\alpha$ with 6 levels and counts $n_\alpha=(u_0+h,u_1,u_0,u_1+h,u_0,u_1)$ for given $u_0,u_1,h \in {\mathbb N}$. It is straightforward to check that $a_\alpha=0$, even though  $X^\alpha$ is not balanced. Moreover, permuting the first two counts $(u_1,u_0+h,u_0,u_1+h,u_0,u_1)$, Equation \eqref{calpha} yields $n^2 a_\alpha = ((u_0-u_1)+h)^2$, while permuting the first and the third counts, $(u_0,u_1,u_0+h,u_1+h,u_0,u_1)$, one obtains $n^2a_\alpha = 3h^2$. To show the relevance of the position of the counts, such three configurations of counts are depicted on the unit circle in Figure \ref{fig:six}.

\begin{figure}\label{fig:six}
\begin{center}
\includegraphics[width=12 cm]{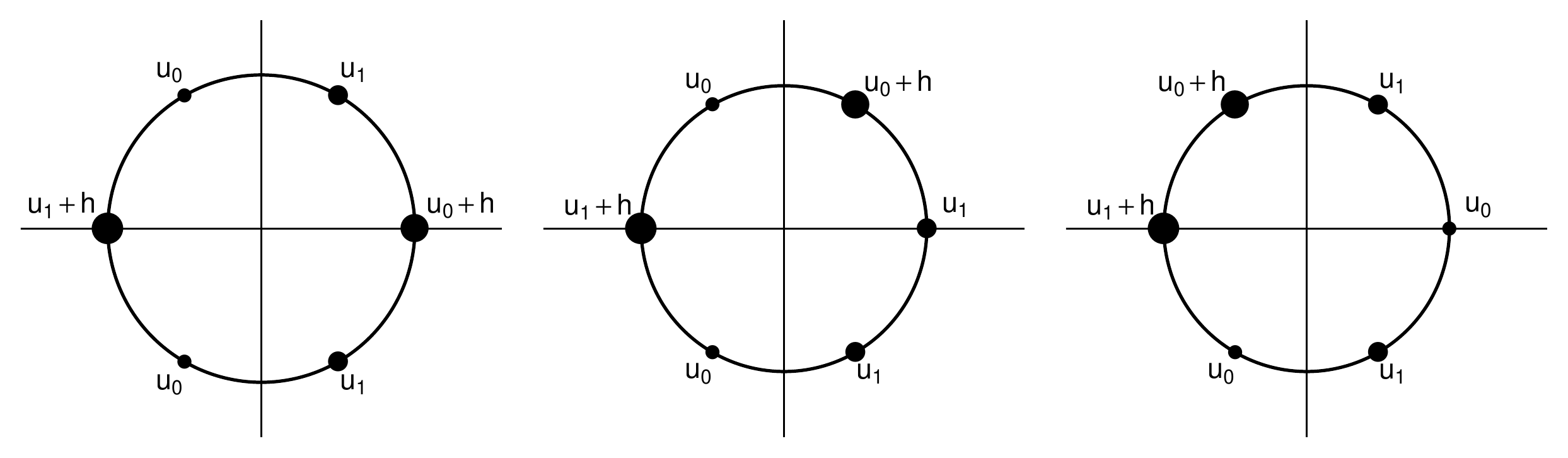}
\end{center}
\caption{Three permutations of counts for a 6-level term. Left:  $a_\alpha=0$; center:  $ a_\alpha = ((u_0-u_1)+h)^2/n^2$; right:   $a_\alpha = 3h^2/n^2 $}
\end{figure}

\subsection{Mean aberration}

As a consequence of the previous discussion, we define here a permutation-invariant aberration of a term $X^\alpha$ as the mean of the aberration obtained through all the possible permutations of the counts $\left\{n_0,\ldots,n_{t-1}\right\}$ of $X^\alpha$.

\begin{definition}
Given a fraction $\mathcal F$ of $\design$ let $X^\alpha$ be a simple or interaction term with $t$ levels and be $(n_{\alpha,0},\dots, n_{\alpha,t-1})$ the number of occurrences of its levels $(0,1,...,t-1)$ respectively. The \emph{mean aberration} of $X^\alpha$ is the mean of all the $a_\alpha$'s obtained by permuting $(n_{\alpha,0},\dots, n_{\alpha,t-1})$:
\begin{equation} \label{properties}
\overline{a}_\alpha = \frac 1 {t!} \sum_{\pi(n_{\alpha,0},\dots, n_{\alpha,t-1}) } a_\alpha \ .
\end{equation}
where $\pi(n_{\alpha,0},\dots, n_{\alpha,t-1})$ denotes all the permutations of the counts of $X^\alpha$.
\end{definition}
The subscript $\alpha$ in $n_{\alpha,i}$ will be omitted in the following when unnecessary.

\begin{proposition} \label{prop:mean_ab}
The mean aberration is:
\begin{equation}\label{eq:mean_ab}
\overline{a}_\alpha =\frac{1}{n^2}\left(\sum_{i=0}^{t-1} n_i^2 -  \frac{2}{t-1}  \sum_{i=0}^{t-1} \sum_{j=i+1}^{t-1} n_i n_j \right)= \frac{1}{n^2}\frac{1}{t-1}\sum_{i=0}^{t-1}  \sum_{j=i}^{t-1} \left( n_i-n_j \right)^2 \ .
\end{equation}
\end{proposition}
\begin{proof}
Take Equation \eqref{al} and first separate the cases $k=0$ and $k>0$ obtaining:
\begin{equation*}
a_\alpha = \frac 1 {n^2} \sum_{i=0}^{t-1} n_i^2 + \frac 1 {n^2} \sum_{k=1}^{t-1}  \cos\left(\frac{2\pi}{t} k \right) \sum_{i=0}^{t-1} n_{i} n_{[i-k]}  \, .
\end{equation*}
The first addendum is permutation-invariant and therefore it is enough to rewrite the second addendum. For fixed $k \ne 0$, we get
\begin{equation*}
\frac 1 {t!} \sum_{\pi(n_{0},\dots, n_{t-1})} \sum_{i=0}^{t-1} n_{i}n_{[i-k]}
\end{equation*}
where in the summations there are $t(t!)$ addenda. For fixed $i$ and $j$ each monomial $n_in_j$ with $i \ne j$ appears exactly $t(t-1)$ times, so that
\begin{equation*}
\frac 1 {t!} \sum_{\pi(n_{0},\dots, n_{t-1})} \sum_{i=0}^{t-1} n_{i}n_{[i-k]} = \frac {1} {t-1} \sum_{i=0}^{t-1} \sum_{j \ne i, j=1}^{t-1} n_i n_j \ .
\end{equation*}
The first equality in Equation \eqref{properties} is then proved, recalling that $\sum_{k=1}^{t-1} \cos \left(\frac {2\pi} k t\right) = -1$.

The second equality is straightforward.
\end{proof}

\begin{remark}
Notice that if the number of levels of $X^\alpha$  in $\design$ is a prime number, the mean aberration of $X^\alpha$ on a fraction is zero if and only the corresponding aberration is zero:
\begin{equation*}
a_\alpha=0 \Leftrightarrow \overline{a}_\alpha = 0 \Leftrightarrow n_0=n_1=\dots = n_{t-1} \, .
\end{equation*}
It is an easy consequence of the properties of the roots of the unity. On the other hand, if the number of levels  is  not a prime number, then  aberration and  mean aberration have different
behaviors. The mean aberration is zero if and only if all the levels appear equally often:
\begin{equation*}
\overline{a}_\alpha=0 \quad \Leftrightarrow \quad n_0=n_1=\dots = n_{t-1}
\end{equation*}
see the right expression of Equation \ref{eq:mean_ab}.
Whereas this fact is not true for aberration, as explained in examples of Figure \ref{fig:six}, where $t=6$ and the level counts are not equal. In all the three cases
\begin{equation*}
\overline{a}_\alpha=\frac 1 {5 n^2}(9(u_0-u_1)^2+8h^2)
\end{equation*}
while the three aberration are
\begin{equation*}
a_\alpha=0 \, , \qquad a_\alpha = ((u_0-u_1)+h)^2/n^2 \, , \qquad a_\alpha = 3h^2/n^2 \, .
\end{equation*}
\end{remark}

Another nice property of mean aberration is that it represents the variances of the counts $\left\{n_0,\ldots,n_{s-1}\right\}$ as
shown in the following proposition.

\begin{proposition}
Given a fraction $\mathcal F$ of $\design$ let $X^\alpha$ be a simple or interaction term with $t$ levels and with level counts
$\left\{n_0,\ldots,n_{t-1}\right\}$. The mean aberration $\overline{a}_{\alpha}$ of $X^\alpha$ is the variance $\sigma^2$ of $\left\{n_0,\ldots,n_{t-1}\right\}$.
\end{proposition}
\begin{proof}
From Proposition \ref{prop:mean_ab} we get
\begin{eqnarray*}
\overline{a}_\alpha &=& \frac{1}{n^2}\frac{1}{t-1}\sum_{i=0}^{t-1}  \sum_{j=i}^{t-1} \left( n_i-n_j \right)^2
= \frac{1}{n^2}\frac{1}{t-1}\frac{1}{2}\sum_{i=0}^{t-1}  \sum_{j=0}^{t-1} \left( n_i-n_j \right)^2 =\\
&=& \frac{1}{n^2} \frac{1}{t-1} \frac{1}{2} \left(2 t \sum_{i=0}^{t-1} n_i^2 - 2 n^2\right)
= \frac{t^2}{n^2(t-1)} \left( \frac{\sum_{i=0}^{t-1} n_i^2}{t} - \frac{n^2}{t^2} \right) =
 \frac{t^2}{n^2(t-1)} \sigma^2
\end{eqnarray*}
where $\displaystyle \sigma^2=\frac{\sum_{i=0}^{t-1} n_i^2}{t} - \frac{n^2}{t^2}$ is the variance of $\left\{n_0,\ldots,n_{t-1}\right\}$.
\end{proof}

The variance is an index to describe the distribution of counts $(n_{0},n_{1}, \dots,n_{t-1})$. Other indices can be used and, for instance, the authors in \cite{fontanaetal:14a} shown the connections between the aberration and Gini concentration index in the two-factor multi-level case.

\begin{remark}[Class of equivalence of aberrations]
The value of $a_\alpha$ does not change for permutations of all the counts in conjugate positions, e.g., for $t=5$, permuting $n_1$ with $n_4$ and $n_2$ with $n_3$ simultaneously. It follows from Equation \eqref{al}. Moreover the same holds for each of the $t-1$ circular permutations of levels, e.g., for $t=5$, permuting $(n_0,n_1,n_2,n_3,n_4)$ in $(n_1,n_2,n_3,n_4,n_0)$ or $(n_2,n_3,n_4,n_0,n_1)$ or $(n_3,n_4,n_0,n_1,n_2)$ or $(n_4,n_0,n_1,n_2,n_3)$.
Then, there are $(t-1)!/2$ class of permutation with different $a_\alpha$.

We explicitly show that the permutation of only two counts in conjugate position is not enough to preserve the same aberration. Consider two terms $X_1^\alpha$ and $X_2^\alpha$ with $t=5$ levels and counts $(n_0,n_1,n_2,n_3,n_4)$ and $(n_0,n_4,n_2,n_3,n_1)$ respectively, where only $n_1$ and $n_4$ are permuted. Let $a_\alpha^1$ and $a_\alpha^2$ the two corresponding aberrations. The difference between $a_\alpha^1$ and $a_\alpha^2$ is:
\begin{equation*}
2 \left(\cos \frac{2 \pi} 5  - \cos \frac{4 \pi} 5\right) (n_1-n_4)(n_2-n_3) \ .
\end{equation*}
\end{remark}

\subsection{GWLP}

From now on we consider symmetric fractional factorial designs $s^m$ with $s$ prime, where each $X^\alpha$ takes values in ${\Omega}_s$.

Restricting to symmetric designs with a prime number of levels, GWLP of a fraction $\fraction$,  $A_\fraction=(A_1(\fraction),\ldots,A_m(\fraction))$, can be computed using only the mean aberrations.

\begin{proposition}
Let $\fraction$ be a fraction of a $s^m$ design, $s$ prime number. Then $j$-th element of the GWLP is:
\begin{equation*}
A_j(\fraction)= \sum_{ \Vert \alpha \Vert_0 =j}  \overline{a}_{\alpha} \quad j=1,\ldots,m \ .
\end{equation*}
\end{proposition}
\begin{proof}
Let $X^\alpha$ be a term of order $j$ with counts $(n_0, \ldots ,n_{s-1})$. Consider the powers of $X^\alpha$ with $\alpha$ multi-exponents in $L$, namely $X^\alpha, X^{[2\alpha]}, \ldots, X^{[(s-1)\alpha]}$. Observe that, as $s$ is prime, the counts of such powers are all permutations of $(n_0, \ldots ,n_{s-1})$. In more details, these counts  have the following properties.
\begin{itemize}
\item $n_0$ is always in the count associated to $\omega_0$ in all the powers $X^{[h\alpha]}$, $h \in \{1,\dots,s-1\}$.

\item Let $n_i$, $i \ne 0$, be the count associated to $\omega_i$ of $X^\alpha$. $n_i$ is the count associated to $\omega_{[hi]}$  of $X^{[h\alpha]}$ and  to $\omega_{[ki]}$  of $X^{[k\alpha]}$, for $h\ne k \in \{1,\dots,s-1\}$. Since $s$ is prime, this means that, for each $i$, $n_i$ becomes the count associated to a different root of the unity exactly once for all the powers.
\end{itemize}
Therefore, the sum of the aberrations of the $s$ terms above rewrites as:
\begin{equation*}
a_\alpha + a_{[2\alpha]} + \cdots + a_{[(s-1)\alpha]}  = \sum_{\pi(n_{0},\dots, n_{s-1})} \frac 1 {n^2} \left( \sum_{i=0}^{s-1} n_s^2 + \frac 1 {n^2} \sum_{k=1}^{s-1}  \cos\left(\frac{2\pi}{s} k \right) \sum_{i=0}^{s-1} n_{i} n_{[i-k]} \right) \, ,
\end{equation*}
where the permutations $\pi$ range over the $s$ permutations defined above.
As in the proof of Proposition \ref{prop:mean_ab}, the first addendum is permutation-invariant and therefore it suffices to rewrite the second addendum. Now notice that in the sum
\begin{equation*}
 \sum_{\pi(n_{0},\dots, n_{s-1})} \sum_{i=0}^{s-1} n_{i} n_{[i-k]}
\end{equation*}
for all $k \ne 1$ each monomial $n_in_j$ with $i \ne j$ appears exactly once, so that
\begin{equation*}
a_\alpha + a_{[2\alpha]} + \cdots + a_{[(s-1)\alpha]}  = (s-1) \overline a_{\alpha} \, .
\end{equation*}
This completes the proof.
\end{proof}

\begin{remark}
We emphasize again that the advantage of using mean aberrations in place of the classical aberration formula is the computational ease. In fact, it is enough to compute the $s^m$ vectors of counts, one of each term $x^\alpha$, $\alpha \in L$, and then use the terse formula in Equation \eqref{eq:mean_ab}. Moreover, according to the previous discussion, the $s^m$ mean aberrations are equal $s$ by $s$ and this reduces further the computational cost.
\end{remark}

\begin{remark}
In the non prime case the GWLP is not the sum of the mean aberrations, as shown in the following counter-example. Consider a trivial example, namely a design with only one factor $X$ with four levels and, on a fraction with 6 runs $n_1=(1,2,1,2)$. Then $X^{2}$ has 2 levels with $n_2=(2,4)$, and $X^{3}$ has 4 levels with $n_{3}=(1,2,1,2)$. We have:
\begin{equation*}
a_1= a_{3}=0, \ a_{2}=1/9, \qquad \overline a_1= \overline a_{3}=1/27, \ \overline a_{2}=1/9 \ .
\end{equation*}
Then  $A_1 = \sum_{i=0}^3 a_{i}= 1/9$, while the sum of the mean aberrations is $5/27$.
\end{remark}

\section{Case studies} \label{sec:choice}

In this section we present some results concerning the use of mean aberrations to distinguish among different fractions. We mainly focus on the most interesting situation where fractions have the same GWLP. We use some non-isomorphic orthogonal arrays extracted from the complete series of non-isomorphic orthogonal arrays, \cite{schoen2010complete}. We consider examples of orthogonal arrays of strength 2, so that ${A_1({\mathcal F})}= {A_2({\mathcal F})}= 0$ in all cases, and the first interesting aberrations correspond to the terms of order 3.

\subsection{$2$-level orthogonal arrays, $OA(16; 2; 2^{10})$}

There are six non-isomorphic orthogonal arrays of strength 2 with ten 2-level factors and 16 runs. We denote these arrays by $\fraction_i , i=1,\ldots,6$. They have the same GWLP:
\begin{equation*}
(0,0,8,18,16,8,8,5,0,0) \, .
\end{equation*}
For each orthogonal array we compute the $10$ mean aberrations of order 3 and the $210$ aberrations of order 4. The mean aberrations corresponding to the interaction terms of order 3 can be classified as in Table \ref{tab_2_10_3}.

\begin{table}[h]
\caption{Aberrations of the terms of order 3 for the $OA(16; 2; 2^{10})$ example.}
\label{tab_2_10_3}\par
\vskip .2cm
\begin{center}
\begin{tabular}{c|rrrrrr}
$\overline{a}_\alpha$ & $\fraction_1$ & $\fraction_2$ & $\fraction_3$ & $\fraction_4$ & $\fraction_5$ & $\fraction_6$ \\
\hline
0	& 112 & 100 &	100 &	88	& 88	& 88 \\
0.25	& 0 & 	16 & 	16 &	32	& 32 &	32	\\
1	& 8	& 4	& 4	& 0	& 0 &	0	\\
\end{tabular}
\end{center}
\end{table}

Table \ref{tab_2_10_3} shows that the six orthogonal arrays can be clustered into three groups: ${\fraction_1}$, $\{\fraction_2, \fraction_3\}$, and $\{\fraction_4, \fraction_5, \fraction_6\}$. The mean aberrations corresponding to the interaction terms of order 4 can be classified as in Table \ref{tab_2_10_4}. Here the six fractions are now six different distributions. Moreover, one can see that $\fraction_1$ contains a regular fraction with $8$ defining words of order $3$ and $18$ words of order $4$. Finally, we remark that the two values of the GWLP $A_3({\mathcal F}_i)=8$ and $A_4({\mathcal F}_i)=18$ may be recovered from the mean aberrations via a simple weighted sum.

\begin{table}[h]
\caption{Aberrations of the terms of order 4 for the $OA(16; 2; 2^{10})$ example.}
\label{tab_2_10_4}\par
\vskip .2cm
\begin{center}
\begin{tabular}{c|rrrrrr}
$\overline{a}_\alpha$ & $\fraction_1$ & $\fraction_2$ & $\fraction_3$ & $\fraction_4$ & $\fraction_5$ & $\fraction_6$ \\
\hline
0	& 192	& 168	& 180	& 192	& 168	& 180 \\
0.25 & 	0	& 32	& 16	& 0	& 32	& 16 \\
1	& 18 & 	10 & 	14	& 18	& 10 &	14 \\
\end{tabular}
\end{center}
\end{table}

\subsection{$3$-level orthogonal arrays, $OA(18; 2; 3^{7})$}

There are three non-isomorphic orthogonal arrays of strength 2 with seven 3-level factors and 18 runs. We denote these arrays by $\fraction_1$, $\fraction_2$ and $\fraction_3$. They have the same GWLP:
\begin{equation*}
(0,0,22,34.5,27,31,6) \, .
\end{equation*}
For each orthogonal array we compute all the $278$ values of the mean aberrations to the interaction terms of order 3. They can be classified as in Table \ref{tab_3_7}. In this case the distribution of the mean aberrations corresponding to the interaction terms of order 3 is sufficient to distinguish the arrays.

\begin{table}[h]
\caption{Aberrations of the terms of order 3 for the $OA(18; 2; 3^{7})$ example.}
\label{tab_3_7}\par
\vskip .2cm
\begin{center}
\begin{tabular}{c|rrr}
$\overline{a}_\alpha$ & $\fraction_1$ & $\fraction_2$ & $\fraction_3$ \\
\hline
    0 &   134 &    198 &    102 \\
    $0.08\overline{3}$ &     96 &      0 &    144 \\
    0.25 &     48 &     80 &     32
\end{tabular}
\end{center}
\end{table}

\subsection{$5$-level orthogonal arrays, $OA(25; 2; 5^{3})$}

There are two non-isomorphic orthogonal arrays of strength 2 with three 5-level factors and 25 runs. We denote these arrays by $\fraction_1$ and $\fraction_2$. They have the same GWLP:
\begin{equation*}
(0,0,4) \, .
\end{equation*}
For each orthogonal array we compute the $64$ values of the mean aberrations corresponding to the interaction terms of order 3. They can be classified as in Table \ref{tab_5_3}. Also in this case the distribution of the mean aberrations corresponding to the interaction terms of order 3 is sufficient to distinguish the arrays.

\begin{table}[h]
\caption{Aberrations of the terms of order 3 for the $OA(25; 2; 5^{3})$ example.}
\label{tab_5_3}\par
\vskip .2cm
\begin{center}
\begin{tabular}{c|rr}
$\overline{a}_\alpha$ & $\fraction_1$ & $\fraction_2$  \\
\hline
0	 & 12	& 60 \\	
0.04 & 	16	& 0	\\
0.06 & 32	& 0	\\
0.36 &	4	& 0	\\
1	& 0	& 4
\end{tabular}
\end{center}
\end{table}


\section{Future directions} \label{sec:future}

The theory presented in this paper and the examples discussed in the previous section show the relevance of the notion of \emph{mean aberration} to discriminate between designs with the same GWLP, but several problems are still open. We mention here only a couple of directions for future work. First, it should be interesting to generalize the notion of {\emph mean aberration} in order to apply it in the general case, namely for factors with number of levels not prime, and for asymmetric designs. Second, we deem important to explore the connections between the geometric structure of the design points of the fraction and its mean aberration. In fact, some results in this direction have already been achieved, and they are presented in \cite{fontanaetal:14b}, \cite{fontanaetal:14c}, and \cite{fontanaetal:15}. In those papers, saturated fractions and $D$-optimal saturated fractions are described in terms of the circuit basis, a combinatorial object computed from the model matrix. Although aberration and GWLP are defined in a model-free framework, nevertheless we think that the study of the geometry of the fractions will yield new interesting results on aberration, and new insights in design selection.

\vskip .65cm
\noindent
Department of Mathematical Science, Politecnico di Torino, Corso Duca degli Abruzzi 24, 10129 Torino, Italy
\vskip 2pt
\noindent
E-mail: roberto.fontana@polito.it
\vskip 2pt

\noindent
Department of Sciences and Technological Innovation, Universit\`a del Piemonte Orientale, Viale Teresa Michel 11, 15121 Alessandria, Italy
\vskip 2pt
\noindent
E-mail: fabio.rapallo@uniupo.it
\vskip 2pt

\noindent
Department of Mathematics, Universit\`a di Genova, Via Dodecaneso 35, 16146 Genova, Italy
\vskip 2pt
\noindent
E-mail: rogantin@dima.unige.it


\end{document}